\documentclass{amsart}
\usepackage{lmodern} 
\usepackage[utf8]{inputenc}
\usepackage{graphicx}
\usepackage{amsfonts,amsmath,amssymb,amsthm,bm,mathtools,textcomp,url,latexsym}
\usepackage{dsfont,frcursive}
\usepackage[all]{xy}
\usepackage[a4paper]{geometry} 
\usepackage{color}
\usepackage[pdftex]{hyperref}

\hypersetup{
pdftitle={},		
pdfauthor={},		
pdfsubject={},		
pdfnewwindow=true,	
pdfkeywords={},	
plainpages=false,
bookmarksopen=true,
bookmarksnumbered=true,
anchorcolor=grey,
colorlinks=false,	
linkcolor=violet,	
citecolor=green,	
filecolor=blue,		
urlcolor=cyan		
menucolor=black,
}

\newcommand{\cF}{\mathcal{F}}
\newcommand{\cG}{\mathcal{G}}

\def\fF{\mathfrak{F}}
\def\fG{\mathfrak{G}}

\newcommand{\Z}{{\mathbb Z}}

\newcommand{\C}{{\mathbb C}}

\def\rla{\leftrightarrows}

\newtheorem{thm}{Theorem}[section]
\newtheorem{prop}{Proposition}[section]

\newtheorem{cor}{Corollary}[section]

\theoremstyle{definition} 

\newcounter{noteno}
  {\hfill{ }\par\medbreak\end{small}\smallbreak}
\newcounter{exno}
  {\hfill{ }\par\medbreak\end{em}\end{small}\smallbreak}
\newcommand{\bi}{\begin{itemize}}
\newcommand{\ei}{\end{itemize}}
\newcommand{\bd}{\begin{description}}
\newcommand{\ed}{\end{description}}
\newcommand{\be}{\begin{enumerate}}
\newcommand{\ee}{\end{enumerate}}

\newcommand{\dis}{\displaystyle}

\def\bc{\begin{center}}
\def\ec{\end{center}}

\def\Frac{\dis \frac}
\def\Sum{\dis \sum}
\def\Prod{\dis \prod}

\newcommand{\bin}[2]{\genfrac{(}{)}{0pt}{}{#1}{#2}}

\def\ov{\overline}
\def\und{\underline} 
\def\no{\noindent}
\def\l{\left}
\def\r{\right}
\def\b{\big}

\def\s{\smallskip} 
 
\begin{document}

\title{Multiplicate inverse forms of terminating hypergeometric series}
\author{Christian {\sc Lavault}}
\date{\small \today}

\begin{abstract}
The multiplicate form of Gould--Hsu's inverse series relations enables to investigate the dual relations of the Chu--Vandermonde--Gau\ss's, the Pfaff--Saalsch\"utz's summation theorems and the binomial convolution formula due to Hagen and Rothe. Several identitity and reciprocal relations are thus established for terminating hypergeometric series. By virtue of the duplicate inversions, we establish several dual formulae of Chu--Vandermonde--Gau\ss's and Pfaff--Saalsch\"utz's summation theorems in Section~\ref{ChuVanGauss} and~\ref{PfaffSaalsch}, respectively. Finally, the last section is devoted to deriving several identities and reciprocal relations for terminating balanced hypergeometric series from Hagen--Rothe's convolution identity in accordance with the duplicate, triplicate and multiplicate inversions.

\s \no \emph{Key words and phrases:} Generalized hypergeometric serie; Gould--Hsu inverse series relations; Multiplicate inversions; Chu--Vandermonde--Gau\ss's summation formula; Pfaff--Saalsch\"utz's summation theorem; Hagen--Rothe's convolution identity.
\end{abstract}
\maketitle

\section{Introduction and motivations}
The problem of how to invert a combinatorial sum has a long history. Several approaches have been tried, from ingenuity to generating functions, from umbral calculus to hypergeometric series. The following fundamental result was obtained in 1973 by Gould and Hsu~\cite{GoHsu73}. Let $(a_k)$ and $(b_k)$ ($k\ge 0$) be two sequences of complex numbers such that the polynomials 
\[
\phi(x;0) := 1\ \quad \text{and}\ \quad \phi(x;n) := \prod_{k=0}^{n-1} (a_k + xb_k)\]
are distinct from zero for all nonnegative integers $x,\,n$. Then, the pair of reciprocal formulas holds
\begin{align*}
f(n) & =\; \sum_{k=0}^{n} (-1)^k \bin{n}{k} \phi(k;n) g(k),\\
g(n) & =\; \sum_{k=0}^{n} (-1)^k \bin{n}{k} \frac{a_k + kb_k}{\phi(n;k+1)} f(k).
\end{align*}

The applications of these Gould--Hsu inverse series relations to terminating series identities have been explored in full details by Chu in~\cite{Chu93,Chu94}. Further, the duplicate form of these reciprocal formulas is devised by Chu in~\cite{Chu02}, where a large class of identities is established for terminating ${}_5F_4(1)$ series, which are closely related to the evaluation of plane partitions. As stated in the latter paper: in the competition of identity proving, the inversion technique would open up `la Terza Via' (the ``third approach'') between the classical series transformation (Pfaff method) and the modern technological WZ-method (by Wilf and Zeilberger). Pursuing this approach, thirty closed formulae for the terminating ${}_3F_2(4/3)$ series were obtained very recently by Chen and Chu in~\cite{ChenChu13} by applying the Gould--Hsu's inversions to Pfaff--Saalsch\"utz's summation theorem along with four transformations for terminating ``almost'' balanced series and two contiguous relations. 

The goal of the present paper is to push the approach forward and investigate terminating hypergeometric series identities. In the next section, the multiplicate form of Gould--Hsu reciprocal series relations is stated and a rigorous proof is provided. By virtue of the duplicate inversions, we establish quite a lot of dual formulae of Chu--Vandermonde--Gau\ss's and Pfaff--Saalsch\"utz's summation theorems~\cite{Bailey35}, in Section~\ref{ChuVanGauss} and Section~\ref{PfaffSaalsch} respectively. Finally, in the fifth section several identities and reciprocal relations for terminating balanced hypergeometric series are derived from the dual relations of Hagen--Rothe's convolution identity~\cite{Chu10,Gould56} in accordance with the duplicate, triplicate and multiplicate inversions.

Following Bailey~\cite[\S2.1]{Bailey35}, the generalized hypergeometric series (or function) is characterized by upper and lower complex parameters and a single argument. It is defined as a power series
\begin{equation} \label{ghf}
{}_pF_q\l[ \begin{matrix} a_1, a_2,\ldots, a_p \\ b_1, b_2,\cdots, b_q \end{matrix}\;;\; z\r] %
=\; \sum_{k\ge 0} \frac{(a_1)_k (a_2)_k\cdots (a_p)_k}{(b_1)_k (b_2)_k\cdots (b_q)_k} \frac{z^k}{k!}\,,
\end{equation}
where $a_j\in \C$ and $b_j\in \C\setminus \Z_{\le 0}$, so that the series is well defined, and $(x)_0 := 1$, $(x)_n := x(x+1) \cdots (x+n-1)$\ for $n\in \Z_{>0}$ (the Pochhammer symbol or ``rising factorial'' $x^{\ov{n}}$). The notation for ${}_pF_q$ in~\eqref{ghf} is commonly simplified to ${}_pF_q\l[a_1,a_2,\cdots,a_p;b_1,b_2,\cdots,b_q;z\r]$, primarily for $p,\,q\le 4$ and its product and fractional forms will respectively abbreviate to 
\[
[\alpha,\beta, \ldots, \gamma]_n := (\alpha)_n (\beta)_n\cdots (\gamma)_n\ \quad \text{and}\ \quad 
\l[ \begin{matrix} \alpha, \beta, \ldots, \gamma \\ A, B, \ldots, C \end{matrix}\r]_n %
:= \frac{(\alpha)_n (\beta)_n\cdots (\gamma)_n}{(A)_n (B)_n\cdots (C)_n}.\]
When $p\le q$, the ${}_pF_q$-function is entire and the series converges everywhere. When $p > q + 1$, the series converges only for $z = 0$, it is therefore significant only when it terminate. In the case when $p = q + 1$---which we are concerned with in this paper---the ${}_{q+1}F_q$-series converges for $|z| < 1$, and also when $z = 1$ provided that 
$\Re\l(\Sum_{k=1}^{q+1} a_k - \Sum_{k=1}^{q} b_k\r) < 0$, and when $z = -1$ provided that 
$\Re\l(\Sum_{k=1}^{q+1} a_k - \Sum_{k=1}^{q} b_k\r) < 1$. (Outside of the unit disk, the series may be defined through analytic continuation.) For brevety, when $z = 1$ the argument $z$ will simply be omitted in the notation.

\section{Multiplicate form of Gould--Hsu inverse series relations} \label{GoulHsu}
Let $\ell$ be a nonnegative integer. Given two sequences $\{A_{i,j}\}$ and $\{B_{i,j}\}$ of $(\ell+1)$ complex terms each ($i = 0, 1 ,\ldots, \ell$ and $j\in \Z_{\ge 0}$), the corresponding polynomials are now defined as
\[
\Phi(x;n) \;=\; \prod_{i=0}^{\ell} \phi_i\l(x;\l[\frac{i+n}{\ell+1}\r]\r),\]
where $[x]$ stands for the integer part of a real number $x$ and $n\pmod{\ell+1}$ is the remainder of $n$ modulo $(\ell + 1)$ (the values of which are ranging from $0$ to $\ell$) and the polynomials $\phi_i(x;n)$ are given, for $0\le i\le \ell$, by
\[
\phi_i(x,0) := 1\ \quad \text{and}\ \quad \phi_i(x;n) := \prod_{k=0}^{n-1} (A_{i,k} + xB_{i,k})\ %
\quad  \text{for}\ n\in \Z_{>0}.\]
Let $\lambda(k)$ denote the linear factor $A_{\ell-k(\text{mod}\ \ell+1),\l[\frac{k}{\ell+1}\r]} + xB_{\ell-k(\text{mod}\ \ell+1),\l[\frac{k}{\ell+1}\r]}$. Owing to the fact that $n = \Sum_{0\le i\le \ell} \l[\frac{i+n}{\ell+1}\r]$, all the $\Phi(x;n)$s are polynomials of degree $n$ in $x$. Now, dividing the linear factors $a_k + xb_k$ into $\ell + 1$ classes in the Gould--Hsu pair of inverse series relations, we have the tools to show Theorem~\ref{multinvser}.

\begin{thm}\label{multinvser} \emph{ (Multiplicate inverse series relations)}
\begin{subequations} 
\begin{align}
F(n) & =\; \sum_{k=0}^{n} (-1)^k \bin{n}{k} \Phi(k;n) G(k), \label{a}\\
G(n) & =\; \sum_{k=0}^{n} (-1)^k \bin{n}{k} \frac{\lambda(k)}{\Phi(n;k+1)} F(k). \label{b}
\end{align}
\end{subequations}
\end{thm}
\begin{proof}
For each one identity of the form~\eqref{a} or~\eqref{b} in the above inverse pair Eq.~\eqref{a}--Eq.\eqref{b}, there exists one companion dual identity. In other words, one system of equations with $F(n)$ in terms of $G(k)$ can be considered as the (unique) solution of another system with $G(n)$ in terms of $F(k)$, and \emph{vive versa}. Therefore, showing the two-way transformations~\eqref{a} $\rla$~\eqref{b} amounts to verify one of them two: supposing that the relations of 
$G(n)$ in terms of $F(k)$ are valid, we have to verify only the relations of $F(n)$ in terms of $G(k)$.

Now, substitute for the expression of $G(k)$ in terms of $F(k)$ (Eq.~\eqref{b}) into~Eq.~\eqref{a}. According to the well-known binomial identity (the ``trinomial revision'')
\[
\bin{n}{k}\times \bin{k}{i} = \bin{n}{i}\times \bin{n-i}{k-i}\ \qquad (i, k\ \text{integers})\,\]
the expression of $F(n)$ (Eq.~\eqref{a}) simplifies to the double sum
\begin{align} \label{eq:doublesum}
F(n) & = \sum_{0\le k\le n} (-1)^k \bin{n}{k} \Phi(k;n) G(k)\nonumber\\
& = \sum_{k=0}^n (-1)^k \bin{n}{k} \Phi(k;n)\; \sum_{i=0}^k (-1)^i \bin{k}{i} \frac{\lambda(i)}{\Phi(k;i+1)} F(i)\nonumber\\
F(n) & = \sum_{i=0}^n \bin{n}{i} \lambda(i) F(i)\; \sum_{k=i}^n (-1)^{k-i} \bin{n-i}{k-i} \frac{\Phi(k;n)}{\Phi(k;i+1)}\,.
\end{align}
When $n = (\ell + 1)p + q$ with $0\le q\le \ell$, then $\lambda(n) = A_{\ell-q,p} + nB_{\ell-q,p}$. On the other hand, if $i = n$ we get the relation
\[
\frac{\Phi(n;n)}{\Phi(n;n+1)} \;=\; \prod_{i=0}^{\ell} %
\frac{\phi_i\l(n;\l[p+\frac{q+i}{\ell+1}\r]\r)}{\phi_i\l(n;\l[p+\frac{q+i+1}{\ell+1}\r]\r)} %
\;=\; \frac{\phi_{\ell-q}(n;p)}{\phi_{\ell-q}(n;p+1)} \;=\; \frac{1}{\lambda(n)},\]
after simplification by cancellation of respective factors of the numerator and the denominator within the fractional product; except if $i = \ell - q$, in which case there remains only one factor in the denominator. Hence, the double sum reduces to $F(n)$ for any nonnegative integer $i = n$.\par
Finally, let $S(i,n)$ denote the last inner sum with respect to $k$ in Eq.~\eqref{eq:doublesum}: 
\[
S(i,n) =\; \sum_{k=i}^n (-1)^{k-i} \bin{n-i}{k-i} \frac{\Phi(k;n)}{\Phi(k;i+1)}\,.\]
In all cases of $i\ne n$, there remains to show that $S(i,n) = 0$ for $0\le i < n$ to prove that $F(n)$ and Eq.~\eqref{eq:doublesum} are equal. The latter result can be established by means of the finite difference method. Indeed, since the fraction $\Phi(k;n)/\Phi(k;i+1)$ is a polynomial of degree $n - i - 1$ in $k$, the differences of order $n - i$ vanish. This completes the proof of Theorem~\ref{multinvser}.
\end{proof}

\section{The Chu--Vandermonde--Gau\ss's summation formula} \label{ChuVanGauss}
When $\ell = 1$, the factors $a_k + xb_k$ are divided into two distinct classes in the Gould--Hsu inversion theorem. Then, the duplicate inverse relations~\eqref{a}--\eqref{b} in Theorem~\ref{multinvser} is as follows. Let $(a_k)$, $(b_k)$, $(c_k)$, $(d_k)$ ($k\ge 0$) be four sequences of real or complex numbers such that the polynomials defined by
\begin{align*}
\phi(x;0) := 1\ \quad \text{and}\ \quad \phi(x;n) & := \prod_{k=0}^{n-1} (a_k + xb_k)\ %
\text{for}\ n\in \Z_{>0}\\
\psi(x;0) := 1\ \quad \text{and}\ \quad \psi(x;n) & := \prod_{k=0}^{n-1} (c_k + xd_k)\ %
\text{for}\ n\in \Z_{>0}
\end{align*}
differ from zero for all nonnegative integers $x,\,n$. Then, the identities (or the system of equations)

\begin{equation} \label{sysF}
\cF(n) = \sum_{0\le k\le n} (-1)^k \bin{n}{k} \phi\l(k;\l[\frac{n}{2}\r]\r) \psi\l(k;\l[\frac{n+1}{2}\r]\r)\; \cG(k) 
\end{equation}
is equivalent to the system

\begin{equation} \label{sysG}
\cG(n) = \sum_{k\ge 0} \bin{n}{2k} \frac{c_k+2kd_k}{\phi(n;k) \psi(n;k+1)}\; \cF(2k) %
\;-\; \sum_{k\ge 0} \bin{n}{2k+1} \frac{a_k+(2k+1)b_k}{\phi(n;k+1) \psi(n;k+1)}\; \cF(2k+1).
\end{equation}

Recall the Chu--Vandermonde--Gau\ss's formula ${}_2F_1[-n,a;c;1] = \Frac{(c-a)_n}{(c)_n}$ (see e.g. Bailey~\cite[\S1.3]{Bailey35}). By applying the duplicate inversions to the dual relations of the formula, several reciprocal relations can be established for terminating hypergeometric series of higher order. Though each of the ${}_{q+1}F_q$-series (with $3 \le q\le 5$) obtained in Subsections~\ref{cvgrecirel1} to~\ref{cvgrecirel4} is zero-balance without closed form, their pairwise combination obviously has. This is the reason why each of the theorems may be regarded in itself as a \emph{reciprocal relation} by providing a pair of terminating hypergeometric series (see~\cite{Wang11}).

\subsection{First type of reciprocal relations for terminating balanced series} \label{cvgrecirel1} \hfill \par
Consider the following alternative form of Chu--Vandermonde--Gau\ss\ formula,
\[
{}_2F_1\b[-n, c - a; c - [n/2]; 1\b] \;=\; \frac{\b(a-[n/2]\b)_n}{\b(c-[n/2])\b)_n} %
\;=\; \l[ \begin{matrix} 1 -a\\ 1 - c \end{matrix}\r]_{[n/2]}\; %
\l[\begin{matrix} a\\ c \end{matrix}\r]_{\l[\frac{n+1}{2}\r]}.\]
It can be expressed as the binomial sum
\[
\sum_{0\le k\le n} (-1)^k \bin{n}{k} (1 - c - k)_{[n/2]} \frac{(c - a)_k}{(c)_k} %
\;=\; \frac{(1 - a)_{[n/2]}(a)_{\l[\frac{n+1}{2}\r]}}{(c)_{\l[\frac{n+1}{2}\r]}}\,.\]
Now, specifying $\phi(x;n)$ to $(1 - c - x)_n$ and $\psi(x;n)$ to $1$ in Eq.~\eqref{sysF}, the latter formula and Eq.~\eqref{sysF} are equivalent. The dual relation corresponding to~\eqref{sysG} can be written
\[
\frac{(c-a)_n}{(c)_n} \;=\; \sum_{k\ge 0} \bin{n}{2k} \frac{(1-a)_k (a)_k}{(1-c-n)_k (c)_k} %
\;+\; \sum_{k\ge 0} \bin{n}{2k+1} \frac{(c+k) (1-a)_k (a)_{k+1}}{(1-c-n)_{k+1} (c)_{k+1}}\,,\]
and it is reformulated in terms of hypergeometric series as

\begin{thm}\label{recirel1} \emph{(Reciprocal relation)}
\begin{equation*} \arraycolsep=2pt
\frac{(c-a)_n}{(c)_n} \;=\; {}_4F_3\l[ \begin{matrix} -n/2, &\frac{1-n}{2}, & a, & 1 - a\\ %
1/2, & c, & 1 - c - n\end{matrix}\r] \;+\; \frac{na}{1-c-n}\; %
{}_4F_3\l[ \begin{matrix} \frac{1-n}{2}, & \frac{2-n}{2}, & 1 + a, & 1 - a\\ 3/2, & c, & 2 - c - n\end{matrix}\r].
\end{equation*}
\end{thm}

Both following identities are proved by Chu and Xei in~\cite{ChuWei08} by means of Legendre inversions.
\begin{cor} \label{cor:chuwei1}
\[ \arraycolsep=2pt
{}_4F_3\l[ \begin{matrix} -n/2, & \frac{1-n}{2}, & a, & 1 - a\\ %
1/2, & c, & 1 - c - n\end{matrix}\r] \;=\; \frac{(c-a)_n}{2(c)_n}\,.\]
\end{cor}
\begin{proof}
Change the sign of $a$ in Theorem~\ref{recirel1}. Upon combining the two relations $\Frac{(c\pm a)_n}{(c)_n}$ obtained that way, the first identity on terminating balanced series in Chu--Wei~\cite[Eq.~5.1a]{ChuWei08} is recovered.
\end{proof}

\begin{cor} \label{cor:chuwei2}
\[ \arraycolsep=2pt
{}_4F_3\l[ \begin{matrix} -n/2, &\frac{1-n}{2}, & a, & 1 - a\\ 1/2, & c, & 1 - c - n\end{matrix}\r] %
\;=\; \frac{(c-a)_{n+1} - (c+a-1)_{n+1}}{(n+1)(1-2a)(c)_n}\,.\]
\end{cor}
\begin{proof}
Similarly, by changing the parameter $a$ into $1 - a$ and shifting $n$ to $n + 1$, the two equations displayed from Theorem~\ref{recirel1} yield the other identity on terminating balanced series in~\cite[Eq.~5.1b]{ChuWei08}.
\end{proof}\par

\subsection{Second type of reciprocal relations for terminating balanced series} \label{cvgrecirel2} \hfill \par
Next, consider the other equivalent form of the Chu--Vandermonde--Gau\ss's summation formula,
\[ \delimitershortfall-1pt
{}_2F_1\l[-n, c - a; c - \l[\frac{n+1}{2}\r]; 1\r]\;=\; \frac{\l(a-\l[\frac{n+1}{2}\r]\r)_n}{\l(c-\l[\frac{n+1}{2}\r]\r)_n} %
\;=\; \l[ \begin{matrix} 1 - a \\ 1 - c \end{matrix}\r]_{\l[\frac{n+1}{2}\r]}\; %
\l[ \begin{matrix} a \\ c \end{matrix}\r]_{[n/2]}.\]
It can be reformulated as another similar binomial sum
\[
\sum_{0\le k\le n} (-1)^k \bin{n}{k} (1 - c - k)_{\l[\frac{n+1}{2}\r]} \frac{(c - a)_k}{(c)_k} %
\;=\; \frac{(a)_{[n/2]} (1 - a)_{\l[\frac{n+1}{2}\r]}}{(c)_{[n/2}]}\;.\]
Upon interchanging the values of $\phi(x;n)$ and $\psi(x;n)$ in Subsection~\ref{cvgrecirel1} (i.e. $\phi(x;n) := 1$ and 
$\psi(x;n) := (1 - c - x)_n$), the dual relation corresponding to~\eqref{sysG} writes
\[
\frac{(c-a)_n}{(c)_n} \;=\; \sum_{k\ge 0} \bin{n}{2k} \frac{(1-c-k) (1-a)_k (a)_{k}}{(1-c-n)_{k+1} (c)_{k}} %
\;-\; \sum_{k\ge 0} \bin{n}{2k+1} \frac{(1-a)_{k+1} (a)_k}{(1-c-n)_{k+1} (c)_k}\,,\]
which can express as the following hypergeometric series relation.

\begin{thm}\label{recirel2} \emph{(Reciprocal relation)}
\begin{align*} \arraycolsep=2pt
\frac{(c-a)_n}{(c)_n} & \;=\; \frac{1-c}{1-c-n}\; %
{}_4F_3\l[ \begin{matrix} -n/2, &\frac{1-n}{2}, & a, & 1 - a\\ 1/2, & c - 1, & 2 - c - n\end{matrix}\r]\\
& -\; \frac{n(1-a)}{1-c-n}\; {}_4F_3\l[ \begin{matrix} \frac{1-n}{2}, & \frac{2-n}{2}, & a, & 2 - a\\ %
3/2, & c, & 2 - c - n\end{matrix}\r].
\end{align*}
\end{thm}

\begin{cor} \emph{(Balanced series identities, $n > 0$)}
\[ \arraycolsep=2pt
{}_5F_4\l[ \begin{matrix} -n/2, & \frac{1-n}{2}, & 1 - \frac{an}{2a+2c+n-2}, & a, & -a\\ %
1/2, & \frac{-an}{2a+2c+n-2}, & c, & 2 - c -n \end{matrix}\r] \;=\; \frac{(c-a)_{n-1}}{(c)_{n-1}}\,.\]
\end{cor}
\begin{proof}
Change $a$ into $a + 1$ in Theorem~\ref{recirel2}. Combining the new relation with Theorem~\ref{recirel1} yields the expected identity of the corollary.
\end{proof}

\begin{cor} \emph{(Balanced series identities)}
\[ \arraycolsep=2pt
{}_5F_4\l[ \begin{matrix} -n/2, & \frac{1-n}{2}, & 1 + \frac{c+n-an}{2a+2c+n-1}, & a, & 1 - a \\ %
3/2, & \frac{c+n-an}{2a+2c+n-1}, & c + 1, & 1 - c - n\end{matrix}\r] %
\;=\; \frac{c+n}{c+n-an}\, \frac{(c+1-a)_{n}}{(c+1)_n}\,.\]
\end{cor}
\begin{proof}
Similarly, change $c$ into $c + 1$ in Theorem~\ref{recirel2} and combine with Theorem~\ref{recirel1}. The proof of the corollary is completed.
\end{proof}

\subsection{Third type of reciprocal relations for terminating balanced series} \label{cvgrecirel3} \hfill \par
The Chu--Vandermonde--Gau\ss's summation formula can also be rewritten as
\begin{equation} \label{eq:cvg3}
{}_2F_1[-n, c - a + [n/2]; c; 1] \;=\; \frac{\b(a-[n/2]\b)_n}{(c)_n} \;=\; %
(-1)^{[n/2]}\; \frac{(1-a)_{[n/2]} (a)_{\l[\frac{n+1}{2}\r]}}{(c)_n}\,,
\end{equation}
which reformulates as the binomial sum
\[
\sum_{0\le k\le n} (-1)^k \bin{n}{k} (c - a + k)_{[n/2]} \frac{(c-a)_k}{(c)_k} %
\;=\; (-1)^{[n/2]}\; \frac{(1 - a)_{[n/2]} (c-a)_{[n/2]} (a)_{\l[\frac{n+1}{2}\r]}}{(c)_{n}}\;.\]
Specify $\phi(x;n)$ to $(c - a + x)_n$ and $\psi(x;n)$ to $1$ in Eq.~\eqref{sysF}. The new identity coincides with 
Eq.~\eqref{sysF} and the dual relation corresponding to~\eqref{sysG} can be written now as
\begin{align*}
\frac{(c-a)_n}{(c)_n} & = \sum_{k\ge 0} \bin{n}{2k} \frac{(-1)^k (1-a)_k (c-a)_k (a)_{k}}{(c-a+n)_{k} (c)_{k}}\\ 
& -\; \sum_{k\ge 0} \bin{n}{2k+1} \frac{(-1)^k (1-c-a+3k) (1-a)_k (c-a)_k (a)_{k+1}}{(c-a+n)_{k+1} (c)_{2k+1}}\,.
\end{align*}
This equation yields the following hypergeometric series relations.

\begin{thm}\label{recirel3a} \emph{(Reciprocal relation)}
\begin{flalign*} 
\frac{(c-a)_n}{(c)_n} & =\; {}_5F_4\l[ \begin{matrix} -n/2, &\frac{1-n}{2}, & c - a, & a, & 1 - a\\ %
c/2, & \frac{c+1}{2}, & 1/2, & c - a + n\end{matrix}\;;\; \frac{-1}{4}\r]\\
& -\; \frac{na(c-a+1)}{c(c-a+n)}\; \arraycolsep=1.4pt %
{}_6F_5\l[ \begin{matrix} \frac{1-n}{2}, & \frac{2-n}{2}, & 1 + \frac{1+c-a}{3}, & c - a, & 1 + a, & 1 - a\\ %
\frac{c+1}{2}, & \frac{c+2}{2}, & 3/2, & \frac{1+c-a}{3}, & c + 1 - a + n\end{matrix}\;;\; \frac{-1}{4}\r].
\end{flalign*}
\end{thm}

Consider again the Chu--Vandermonde--Gau\ss's summation formula~\eqref{eq:cvg3}. From the variant
\[ \delimitershortfall-1pt
{}_2F_1\l[-n, c - a + \l[\frac{n+1}{2}\r]; c; 1\r] \;=\; \frac{\b(a-\l[\frac{n+1}{2}\r]\b)_n}{(c)_n} \;=\; %
(-1)^{\l[\frac{n+1}{2}\r]}\; \frac{(a)_{[n/2]} (1-a)_{\l[\frac{n+1}{2}\r]}}{(c)_n}\,,\]
the same above procedure results in another hypergeometric series identity.

\begin{thm}\label{recirel3b} \emph{ Reciprocal relation)}
\begin{align*} \arraycolsep=2pt
\frac{(c + 1-a)_n}{(c)_n} & =\; %
{}_6F_5\l[ \begin{matrix} -n/2, & \frac{1-n}{2}, & 1 + \frac{c-a}{3}, & c - a, & a, & 1 - a\\ %
c/2, & \frac{c+1}{2}, & 1/2, & \frac{c-a}{3}, & c + 1 - a + n\end{matrix}\;;\; \frac{-1}{4}\r]\\
& -\; \frac{n(1-a)}{c}\; {}_5F_4\l[ \begin{matrix} \frac{1-n}{2}, & \frac{2-n}{2}, & c + 1 - a, %
& a, & 2 - a \\ \frac{c+1}{2}, & \frac{c+2}{2}, & 3/2, & c + 1 - a + n\end{matrix}\;;\; \frac{-1}{4}\r].
\end{align*}
\end{thm}

\subsection{Fourth type of reciprocal relations for terminating balanced series} \label{cvgrecirel4} \hfill \par
According to Chu--Vandermonde--Gau\ss\ summation formula
\[ \delimitershortfall-1pt
{}_2F_1\l[-n, c - a + [n/2]; c - \l[\frac{n+1}{2}\r]; 1\r] =\, \frac{(a-n)_n}{\l(c-\l[\frac{n+1}{2}\r]\r)_n} %
\;=\; (-1)^{[n/2]}\; \frac{(1-a)_{n}}{(c)_{[n/2]} (1-c)_{\l[\frac{n+1}{2}\r]}}\,,\]
we have the binomial sum
\[
\sum_{0\le k\le n} (-1)^k \bin{n}{k} (c - a + k)_{[n/2]} (1 - c - k)_{\l[\frac{n+1}{2}\r]}\;\frac{(c-a)_k}{(c)_k} %
\;=\; (-1)^{[n/2]} (1-a)_n\; \frac{(c-a)_{[n/2]}}{(c)_{[n/2}]}\;.\]
This equation is equivalent to Eq~.\eqref{sysF} upon specifying $\phi(x;n)$ to $(c - a + x)_n$ and $\psi(x;n)$ to $(1 - c - x)_n$; the dual relation corresponding to~\eqref{sysG} is given by

\begin{align*}  
\frac{(c-a)_n}{(c)_n} & = %
\sum_{k\ge 0} \bin{n}{2k} \frac{(-1)^k (1-c-k) (1-a)_{2k} (c-a)_{k}}{(c-a+n)_{k} (1-c-n)_{k+1} (c)_{k}}\\
& - \sum_{k\ge 0} \bin{n}{2k+1} \frac{(-1)^k (1+c-a+3k) (1-a)_{2k+1} (c-a)_k}{(c-a+n)_{k+1}(1-c-n)_{k+1} (c)_k}\,,
\end{align*}
which yields the following hypergeometric series identity.

\begin{thm}\label{recirel4a} \emph{ (Reciprocal relation)}
\begin{flalign*}  \arraycolsep=2pt
\frac{(c-a)_n}{(c)-1_n} & = %
{}_5F_4\l[ \begin{matrix} -n/2, & \frac{1-n}{2}, & \frac{1-a}{2}, & \frac{2-a}{2}, & c - a\\ 
1/2, & c - 1, & 2 - c - n, & c - a + n \end{matrix}\;;\; -4\r] &\\
& \kern-.5cm -\; \frac{n(1-a)(c-a+1)}{(1-c)(c-a+n)}\; \arraycolsep=1.3pt %
{}_6F_5\l[ \begin{matrix} 1 + \frac{1+c-a}{3}, & \frac{1-n}{2}, & \frac{2-n}{2}, & \frac{2-a}{2}, & \frac{3-a}{2}, \ c-a\\ %
3/2, & \frac{1+c-a}{3}, & c, & 2 - c - n, & c + 1 - a + n\end{matrix}\;;\, -4\r]. &
\end{flalign*}
\end{thm}

Finally, from the Chu--Vandermonde--Gau\ss's convolution formula
\[ \delimitershortfall-1pt
{}_2F_1\l[-n, c - a + \l[\frac{n+1}{2}\r]; c - [n/2]; 1\r] \,=\; \frac{(a-n)_n}{(c-[n/2])_n} %
\;=\; (-1)^{\l[\frac{n+1}{2}\r]}\; \frac{(1-a)_{n}}{(1-c)_{[n/2]} (c)_{\l[\frac{n+1}{2}\r]}}\,,\]
we have in similarly the following hypergeometric reciprocal relations

\begin{thm}\label{recirel4b}
\begin{align*} \arraycolsep=1.4pt
\frac{(c - a + 1)_n}{(c)_n} & = %
{}_6F_5\l[ \begin{matrix} 1 + \frac{c-a}{3}, & -n/2, & \frac{1-n}{2}, & \frac{1-a}{2}, & \frac{2-a}{2}, & c - a\\ 
1/2, & \frac{c-a}{3}, & c, & 1 - c - n, & c + 1 - a + n\end{matrix}\;;\; -4\r]\\
& -\; \frac{n(1-a)}{1 - c - n}\; %
{}_5F_4\l[ \begin{matrix} \frac{1-n}{2}, & \frac{2-n}{2}, & \frac{2-a}{2}, & \frac{3-a}{2}, & 1 + c - a\\ %
3/2, & c, & 2 - c - n, & c + 1 - a + n\end{matrix}\;;\; -4\r]. 
\end{align*}
\end{thm}

\section{The Pfaff--Saalsch\"utz's summation theorem} \label{PfaffSaalsch}
Recall the Pfaff--Saalsch\"utz's summation formula (see e.g. Bailey~\cite[\S2.2]{Bailey35}
\[
{}_3F_2[-n, a, b; c, 1 + a + b - c - n; 1] = \l[ \begin{matrix} c - a, & c - b \\ c, & c - a - b\end{matrix}\r]_n.\]
Along the same lines as in Subsections~\ref{cvgrecirel1} to~\ref{cvgrecirel4}, the duplicate inversions applied to the dual relations of the formula leads to four reciprocal relations for terminating hypergeometric series of higher order.

\subsection{First type of reciprocal relations for terminating balanced series} \label{psreciprel1} \hfill \par
The following equivalent form of Pfaff--Saalsch\"utz's summation formula is easy to check,
\begin{flalign*} \delimitershortfall-1pt
{}_3F_2\l[ -n, c - a, a - b; c - [n/2], 1 - b - \l[\frac{n+1}{2}\r]; 1\r] & =\; %
\frac{(a-[n/2])_n (b+c-a-[n/2])_n}{(b-[n/2])_n (c-[n/2])_n} &\\
& \kern-2cm =\; \l[ \begin{matrix} 1 - a, & 1 + a - b - c\\ 1 - b, & 1 - c \end{matrix} \r]_{[n/2]}\; %
\l[ \begin{matrix} a, & b + c - a\\ b, & c \end{matrix} \r]_{\l[\frac{n+1}{2}\r]}. &
\end{flalign*}
This formula can be restated as the binomial sum

\begin{align*} \arraycolsep=1.4pt
\sum_{0\le k\le n} (-1)^k \bin{n}{k} (1 - c - k)_{[n/2]} (b-k)_{\l[\frac{n+1}{2}\r]} %
\l[ \begin{matrix} c - a, & a - b\\ c, & 1 - b \end{matrix} \r]_{k} &\\
\kern2cm \;=\; \arraycolsep=1.4pt \l[ \begin{matrix} 1 - a, & 1 + a - b - c\\ 1 - b \end{matrix} \r]_{[n/2]}\; %
\l[ \begin{matrix} a, & b + c - a\\ c \end{matrix} \r]_{\l[\frac{n+1}{2}\r]}, &
\end{align*}
which is equivalent to Eq~.\eqref{sysF} by specifying $\phi(x;n)$ to $(1 - c - x)_n$ and $\psi(x;n)$ to $(b - x)_n$. The dual relation corresponding to~\eqref{sysG} results then in

\begin{align*}
\l[ \begin{matrix} c - a, & a - b\\ c, & 1 - b \end{matrix} \r]_{n} = %
\sum_{k\ge 0} \bin{n}{2k} \frac{(b-k) (1-a)_{k} (a)_{k}}{(1-c-n)_{k} (b-n)_{k+1}} %
\frac{(a-b-c+1)_{k} (b+c-a)_{k}}{(1-b)_{k} (c)_k} &\\
+\; \sum_{k\ge 0} \bin{n}{2k+1} \frac{(c+k) (1-a)_{k} (a)_{k+1}}{(1-c-n)_{k+1} (b-n)_{k+1}} %
\frac{(a-b-c+1)_{k} (b+c-a)_{k+1}}{(1-b)_{k} (c)_{k+1}} &\,.
\end{align*}

In terms of hypergeometric series, the relation writes  
\begin{thm}\label{psrecirel1} \emph{(Reciprocal relation)}
\begin{flalign*} \arraycolsep=1.4pt
\l[ \begin{matrix} c - a, & a - b\\ c, & - b \end{matrix} \r]_{n} & =\; %
{}_6F_5\l[ \begin{matrix} -n/2, & \frac{1-n}{2}, & a, & 1 - a, & b + c - a, & a - b - c + 1\\ 
1/2, & c, & 1 - c - n, & b - n + 1, & -b \end{matrix}\r] &\\
& \kern-1.2cm -\; \frac{na(b + c - a)}{b(c + n - 1)}\; \arraycolsep=1.4pt %
{}_6F_5\l[ \begin{matrix} \frac{1-n}{2}, & \frac{2-n}{2}, & a + 1, & 1 - a, & b + c - a + 1, & a - b - c + 1\\ %
3/2, & c, & 2 - c - n, & b - n + 1, & 1 - b\end{matrix}\r]. &
\end{flalign*}
\end{thm}

\subsection{Second type of reciprocal relations for terminating balanced series} \label{psreciprel2} \hfill \par
The Pfaff--Saalsch\"utz's summation theorem can also be restated in the form
\begin{align*} \delimitershortfall-1pt
& {}_3F_2\l[ -n, c - a + [n/2], a - b; c - \l[\frac{n+1}{2}\r], 1 - b ; 1\r] =\; %
\frac{(1 - a)_n \l(b + c - a -\l[\frac{n+1}{2}\r]\r)_n}{(1 - b)_n \l(c - \l[\frac{n+1}{2}\r]\r)_n} &\\
& \kern4cm \;=\; \l[ \begin{matrix} 1 - a \\ 1 - b \end{matrix} \r]_n\; %
\l[ \begin{matrix} b + c - a \\ c \end{matrix} \r]_{[n/2]}\; %
\l[ \begin{matrix} a - b - c + 1 \\ 1 - c \end{matrix} \r]_{\l[\frac{n+1}{2}\r]}, &
\end{align*}
which in turns is rewritten as the binomial sum

\begin{align*} \arraycolsep=1.4pt
& \sum_{0\le k\le n} (-1)^k \bin{n}{k} (c - a + k)_{[n/2]} (1-c-k)_{\l[\frac{n+1}{2}\r]}\; %
\l[ \begin{matrix} c - a, & a - b\\ c, & 1 - b \end{matrix} \r]_{k} &\\
& \kern2cm \;=\; \l[ \begin{matrix} 1 - a\\ 1 - b \end{matrix} \r]_{n}\; %
\l[ \begin{matrix} b + c - a \\ c \end{matrix} \r]_{[n/2]}\; (c - a)_{[n/2]}\; (a - b - c + 1)_{\l[\frac{n+1}{2}\r]}. &
\end{align*}

As in Subsection~\ref{cvgrecirel3}, take now $\phi(x;n) = (c - a + x)_n$ and $\psi(x;n) = (1 - c - x)_n$. Comparing with Eq~.\eqref{sysF}, we have the dual relation related to~\eqref{sysG},

\begin{align*} \arraycolsep=1.3pt
\l[ \begin{matrix} c - a, & a - b\\ c, & 1 - b \end{matrix} \r]_{n} =\; %
\sum_{k\ge 0} \bin{n}{2k} \frac{(1-c-k) (a-b-c+1)_{k}}{(c-a+n)_{k} (1-c-n)_{k+1}}\; %
\l[ \begin{matrix} 1 - a \\ 1 - b \end{matrix} \r]_{2k}\; \frac{[c-a,b+c-a]_{k}}{(c)_k} &\\
+\; \sum_{k\ge 0} \bin{n}{2k+1} \frac{(c-a+3k+1) (a-b-c+1)_{k+1}}{(c-a+n)_{k+1} (1-c-n)_{k+1}}\; %
\l[ \begin{matrix} 1 - a \\ 1 - b \end{matrix} \r]_{2k+1}\, \frac{[c-a,b+c-a]_{k}}{(c)_k} &\,,
\end{align*}
which results in

\begin{thm}\label{psrecirel2a} \emph{(Reciprocal relation)}
\begin{flalign*} \arraycolsep=1.5pt
\l[ \begin{matrix} c - a, & a - b \\ c - 1, & 1 - b \end{matrix} \r]_{n} & =\; \arraycolsep=1.5pt %
{}_7F_6\l[ \begin{matrix} -n/2, & \frac{1-n}{2}, & \frac{1-a}{2}, & \frac{2-a}{2}, & c - a, & b + c - a, & a - b - c + 1\\ 
1/2, & \frac{1-b}{2}, & \frac{2-b}{2}, & c - 1, & 2 - c - n, & c - a + n \end{matrix}\r] &\\
& -\; \frac{n(1-a)(c-a+1)(a-b-c+1)}{(1-b)(1-c)(c-a+n)} &\\
\times &\; {}_8F_7\l[ \begin{matrix} \frac{1-n}{2}, \ \frac{2-n}{2}, \ 1 + \frac{c-a+1}{3}, %
\ \frac{2-a}{2}, \ \frac{3-a}{2}, \ c - a, \ b + c - a, \ a - b - c + 2\\
3/2, \quad \frac{c-a+1}{3}, \quad \frac{2-b}{2}, \quad \frac{3-b}{2}, \ c, %
\quad 2 - c - n, \quad c - a + n + 1 \end{matrix}\r]. &
\end{flalign*}
\end{thm}

Alternatively, the next form of the Pfaff--Saalsch\"utz's summation formula
\begin{flalign*} \delimitershortfall-1pt
{}_3F_2\l[ -n, c - a + \l[\frac{n+1}{2}\r], a - b; c - [n/2], 1 - b;1 \r] & =\; %
\frac{(1 - a)_n (b + c - a -[n/2])_n}{(1 - b)_n (c - [n/2])_n} &\\
& \kern-2cm =\; \l[ \begin{matrix} 1 - a \\ 1 - b \end{matrix} \r]_n \; %
\l[ \begin{matrix} a - b - c + 1 \\ 1 - c \end{matrix} \r]_{[n/2]}\; %
\l[ \begin{matrix} b + c - a \\ c \end{matrix} \r]_{\l[\frac{n+1}{2}\r]} &
\end{flalign*}
yields another identity through the same process,

\begin{thm}\label{psrecirel2b} \emph{(Reciprocal relation)}
\begin{align*} \arraycolsep=1.3pt
& \l[ \begin{matrix} 1 + c - a, & a - b \\ c, & 1 - b \end{matrix} \r]_{n} =\; %
{}_8F_7\l[ \begin{matrix} \frac{-n}{2}, \; \frac{1-n}{2}, \; 1 + \frac{c-a}{3}, \; \frac{1-a}{2}, %
\; \frac{2-a}{2}, \; c - a, \; b + c - a, \; a - b - c + 1\\ 
1/2, \quad \frac{c-a}{3}, \ \frac{1-b}{2}, \ \frac{2-n}{2}, \ c, %
\quad 1 - c - n, \quad c - a + n  + 1 \end{matrix}\r] &\\
& + \frac{n(1-a)(b+c-a)}{(1-b)(1-c-n))}\; %
{}_7F_6\l[\begin{matrix} \frac{1-n}{2},\frac{2-n}{2}, \frac{2-a}{2}, \frac{3-a}{2}, c - a + 1, b + c - a + 1, a - b - c + 1\\
3/2, \quad \frac{2-b}{2}, \quad \frac{3-b}{2}, \quad c, \quad 2 - c - n, \quad c - a + n + 1 \end{matrix}\r]. &
\end{align*}
\end{thm}

\subsection{Third type of reciprocal relations for terminating balanced series} \label{psreciprel3} \hfill \par
Finally, the Pfaff--Saalsch\"utz's summation formula may also take the form
\begin{flalign*} \arraycolsep=1pt
{}_3F_2\l[-n, c - a + [n/2],  a - b + \l[\frac{n+1}{2}\r];  c, 1 - b \r] & \;=\; %
\frac{(a-[n/2])_n \l(b + c - a - \l[\frac{n+1}{2}\r]\r)_n}{(b-n)_n (c)_n} &\\
& \kern-1cm =\; \frac{[1 - a, b + c - a]_{[n/2]}\, [a, a - b - c + 1]_{\l[\frac{n+1}{2}\r]}}{(1-b)_n (c)_n}\,, &
\end{flalign*}
which can be again expressed as the binomial sum
\begin{flalign*} 
& \sum_{0\le k\le n} (-1)^k \bin{n}{k} (c - a + k)_{[n/2]} (a - b - k)_{\l[\frac{n+1}{2}\r]}\; %
\l[ \begin{matrix} c - a, & a - b\\ c, & 1 - b \end{matrix} \r]_{k} &\\
& \kern4cm =\; \frac{\l[c - a, 1 - a, b + c - a\r]_{[n/2]}\, [a - b, a,  a - b - c + 1]_{\l[\frac{n+1}{2}\r]}}{(1-b)_n (c)_n}\,. &
\end{flalign*}
Comparing with Eq~.\eqref{sysF} yields the dual relation corresponding to~\eqref{sysG},

\begin{flalign*} 
& \l[ \begin{matrix} c - a, \, a - b \\ c, \; 1 - b \end{matrix} \r]_{n} =\; %
\sum_{k\ge 0} \bin{n}{2k} \frac{(a-k+3k) (1-a)_{k} (a)_{k}}{(a-b+n)_{k+1} (c-a+n)_{k}}\, %
\frac{[a - b, c - a,  b + c - a, a - b - c + 1]_{k}}{(1 - b)_{2k} (c)_{2k}} &\\
& \kern.3cm -\; \sum_{k\ge 0} \bin{n}{2k+1} \frac{(1+c-a+3k) (1-a)_{k} (a)_{k+1}}{(a-b+n)_{k+1} (c-a+n)_{k+1}}\; %
\frac{[c - a, b - c - a]_{k}\, [a - b, a - b - c + 1]_{k+1}}{(c)_{2k+1} (1 - b)_{2k+1}}\,. &
\end{flalign*}

This gives the following hypergeometric series relation.
\begin{thm}\label{psrecirel3} \emph{(Reciprocal relation)}
\begin{flalign*} 
\l[ \begin{matrix} c - a, a - b + 1 \\ c, 1 - b \end{matrix} \r]_{n} & = %
{}_9F_8\l[ \begin{matrix} \frac{-n}{2}, \frac{1-n}{2}, 1 + \frac{a-b}{3}, a - b, c - a, a, 1 - a, %
b + c - a, a - b - c + 1 \\ 1/2, \ \frac{a-b}{3}, \ \frac{1-b}{2}, \ \frac{2-b}{2}, \ c/2, \ \frac{c+1}{2}, %
\ c - a + n, \ 1 + a - b + n \end{matrix} ; \frac{1}{16}\r] &\\
& \kern-3cm-\; \frac{na (c-a+1)( a-b-c+1)}{(1-b) c (c-a+n)}\; \times &\\
& \kern-2.5cm {}_9F_8\l[ \begin{matrix} \frac{1-n}{2}, \ \frac{2-n}{2}, \ 1 + \frac{c-a+1}{3}, %
\ a + 1, \ a - b + 1, \ c - a, \ 1 - a, \; b + c - a, \ a - b - c + 2\\
3/2, \quad \frac{c-a+1}{3}, \quad \frac{2-b}{2}, \quad \frac{3-b}{2}, \quad \frac{c+1}{2}, %
\quad \frac{c+2}{2}, \quad c - a + n + 1, \quad a - b - c + 1\end{matrix}\;;\; \frac{1}{16}\r]. & 
\end{flalign*}
\end{thm}

\section{The Hagen--Rothe's convolution identity} \label{HagenRothe}
In this section, in addition to the duplicate inversions, we make also use of the triplicate inversions corresponding to the case of $\ell = 2$ in the multiplicate inverse relations~\eqref{a}--\eqref{b} of Theorem~\ref{multinvser}.\par
Let $(a_k)$, $(b_k)$, $(c_k)$, $(d_k)$, $(e_k)$, $(f_k)$ ($k\ge 0$) be six sequences of complex numbers such that the polynomials defined by
\begin{align*}
\phi(x;0) := 1\ \quad \text{and}\ \quad \phi(x;n) & := \prod_{k=0}^{n-1} (a_k + xb_k)\ %
\text{for}\ n\in \Z_{>0}\\
\psi(x;0) := 1\ \quad \text{and}\ \quad \psi(x;n) & := \prod_{k=0}^{n-1} (c_k + xd_k)\ %
\text{for}\ n\in \Z_{>0}\\
\chi(x;0) := 1\ \quad \text{and}\ \quad \chi(x;n) & := \prod_{k=0}^{n-1} (e_k + xf_k)\
\text{for}\ n\in \Z_{>0}
\end{align*}
differ from zero for all nonnegative integers $x,\,n$. Then, the relations

\begin{equation} \label{hrsysF}
\fF(n) = \sum_{0\le k\le n} (-1)^k \bin{n}{k}\; \phi\l(k;\l[\frac{n}{3}\r]\r) %
\psi\l(k;\l[\frac{n+1}{3}\r]\r)\chi\l(k;\l[\frac{n+2}{3}\r]\r)\; \fG(k) 
\end{equation}
are equivalent to the system of equations
\begin{align} \label{hrsysG}
\fG(n) & = \sum_{k\ge 0} \bin{n}{3k}\; \frac{e_k+3f_k}{\phi(n;k) \psi(n;k+1) \chi(n;k+1)}\; \fF(3k)\nonumber\\
& -\; \sum_{k\ge 0} \bin{n}{3k+1}\; \frac{c_k+(3k+1)d_k}{\phi(n;k) \psi(n;k+1) \chi(n;k+1)}\; \fF(3k+1)\nonumber\\
& +\; \sum_{k\ge 0} \bin{n}{3k+2}\; \frac{a_k+(3k+2)b_k}{\phi(n;k+1) \psi(n;k+1) \chi(n;k+1)}\; \fF(3k+2).
\end{align}

The following identities for terminating hypergeometric series are derived in the next three subsections by making use of the duplicate, triplicate and multiplicate inversions to the binomial convolution formula due to Haguen and Rothe. This formula is a generalization of the Chu--Vandermonde's convolution identity (see e.g.~\cite[Eq.~2]{Chu10}, \cite{Gould56} and \cite[\S5.4, Eq.~2.62]{GrKP94}).

\begin{equation} \label{eq:hrconv}
\sum_{k\ge 0} \frac{a}{a+bk}\; \bin{a + bk}{k} \bin{c - bk}{n-k} = \bin{a+c}{n}.
\end{equation}

\subsection{First type of reciprocal relations for terminating balanced series} \label{hrreciprel1} \hfill \par
Substitute $c + \l[\frac{n+1}{2}\r]$ for $c$ into Eq.~\eqref{eq:hrconv} and plug the relation into the trinomial revision identity of Theorem~\ref{multinvser}. We have
\[
\bin{c + \l[\frac{n+1}{2}\r] - bk}{n - k} = (-1)^{[n/2]}\; %
\frac{(bk - k -c)_{[n/2]} (c - bk + 1)_{\l[\frac{n+1}{2}\r]}}{(n-1)!\, (c - bk + 1)_k}\,.\]
This yields an alternative formula to Hagen--Rothe's identity in Eq.~\eqref{eq:hrconv},
\[
\sum_{0\le k\le n} (-1)^k \bin{n}{k} (bk - k - c)_{[n/2]} (c - bk + 1)_{\l[\frac{n+1}{2}\r]}\; \omega(k) %
\;=\; (-1)^{[n/2]} \b(a + c + 1 - [n/2]\b)_n\,,\]
where $\omega(k) :=\, \Frac{a}{a + bk - k}\; \frac{(1 - a - bk)_k}{(c - bk + 1)_k}$.

Take $\phi(x;n) = (bx - x - c)_n$ and $\psi(x;n) = (c + 1 - bx)_n$. The latter identity coincides with Eq.~\eqref{sysF} and the dual relation corresponding to~\eqref{sysG} is

\begin{flalign*}
& \omega(n) \;=\; \sum_{k\ge 0} \bin{n}{2k} \frac{(c-2bk+k+1) (a+c+1)_k (-a-c)_{k}}{(bn-n-c)_{k} (c-bn+1)_{k+1}} &\\ 
& \kern4cm +\; \sum_{k\ge 0} \bin{n}{2k+1} \frac{(c-2bk-b+k+1) (a+c+1)_{k+1} (-a-c)_{k}}{(bn-n-c)_{k+1} (c-bn+1)_{k+1}}\,. &
\end{flalign*}

Upon replacing the parameters $a + b$ by $- 1 - a$ and $c + 2 + n/2$ by $c + 1$, the last equality writes in terms of hypergeometric series.

\begin{thm}\label{hrrecirel1} \emph{(Reciprocal relation)}
\begin{flalign*} 
\frac{c + a + bn}{c + a}\; \frac{(c + a)_n}{(c)_n} & \;=\; \frac{c + bn}{c}\; %
{}_5F_4\l[ \begin{matrix} -n/2, & \frac{1-n}{2}, & 1 + \frac{c+bn}{1+2b}, & -a, & a + 1 \\ 
1/2, & \frac{c+bn}{1-2b}, & c + 1, & 1 - c - n \end{matrix}\r] &\\
& =\; \frac{na (bn - b + c)}{c (c - 1 + n)}\; %
{}_5F_4\l[ \begin{matrix} \frac{1-n}{2}, & \frac{2-n}{2}, & 1 + \frac{c+bn-b}{1-2b}, & 1 - a, & a + 1\\ 
3/2, & \frac{c + bn - b}{1 - 2b}, & c + 1, & 2 - c - n \end{matrix}\r]. &
\end{flalign*}
\end{thm}

\begin{prop} \label{prop:hrbalid1} \emph{(Terminating balanced series identity)}
\begin{flalign*} \arraycolsep=1.3pt
& {}_5F_4\l[ \begin{matrix} -n/2, & \frac{1-n}{2}, & 1 + \frac{c+bn}{1-2b}, & -a, & a\\ 
1/2, & \frac{c+bn}{1-2b}, & c + 1, & 1 - c - n \end{matrix}\r] &\\
& \kern4cm =\; \frac{c (c+a+bn)}{2 (c+bn) (a+c)}\, \frac{(c + a)_n}{(c)_n} %
\;+\; \frac{c (c-a+bn)}{2(c+bn) (c-a)}\, \frac{(c - a)_n}{(c)_n}\,. &
\end{flalign*}
\end{prop}
\begin{proof}
Changing the sign of $a$ in Theorem~\ref{hrrecirel1} and adding the two identities thus obtained, we get the desired proposition on terminating balanced series identity with a free parameter $b$.
\end{proof}
For particular values of $b$, Proposition~\ref{prop:hrbalid1} provides several other identities of interest as special cases. If we let $b = 0$, Corollary~\ref{cor:chuwei1} in Subsection~\ref{cvgrecirel1} is recovered again. Next Corollary~\ref{cor:threecor} exemplifies three corollaries involving ${}_4F_3$-series. They derive from Proposition~\ref{prop:hrbalid1} readily for three specific values of $b$.

\begin{cor} \label{cor:threecor}
\bi
\item[\,]
\item[(i)]\ If we let $b = \frac{2c + 1}{2(1 - n)}$ in Proposition~\ref{prop:hrbalid1}, then
\begin{flalign*} 
& {}_4F_3\l[ \begin{matrix} -n/2, & \frac{1-n}{2}, &-a, &a\\-1/2, & c + 1, & 1 - c - n\end{matrix}\r] &\\
& \kern3cm =\; \frac{c(2c+2a-2na+n)}{2(2c+n)(c+a)}\, \frac{(c+a)_n}{(c)_n} %
\;+\; \frac{c(2c-2a+2na+n)}{2(2c+n)(c-a)}\, \frac{(c-a)_n}{(c)_n}\,.
\end{flalign*}

\item[(ii)]\ If we let $b = \frac{a-c}{2a+n}$ in Proposition~\ref{prop:hrbalid1} (see~\cite[Eq.~5.2a]{ChuWei08}), then
\begin{equation*} 
{}_4F_3\l[ \begin{matrix}-n/2, & \frac{1-n}{2}, & -a, & a + 1\\1/2, & c + 1, & 1 - c - n \end{matrix}\r] %
\;=\; \frac{c}{(2c+n)}\, \frac{(c+a+1)_n + (c-a)_n}{(c)_n}\,.
\end{equation*}

\item[(iii)]\ If we let $b \to 1/2$ in Proposition~\ref{prop:hrbalid1} (see~\cite[Eq.~5.3a]{ChuWei08}), then
\begin{flalign*} 
& {}_4F_3\l[ \begin{matrix} -n/2, & \frac{1-n}{2}, & -a, & a \\ 1/2, & c + 1, & 1 - c - n \end{matrix}\r] &\\
& \kern3cm =\; \frac{c(2c+2a+n)}{2(2c+n)(c+a)}\, \frac{(c+a)_n}{(c)_n} \;+\; \frac{c(2c-2a+n)}{2(2c+n)(c-a)}\, %
\frac{(c-a)_n}{(c)_n}\,.
\end{flalign*}
\ei
\end{cor}

\begin{prop} \label{prop:hrbalid2} \emph{(Terminating balanced series identity)}
\begin{flalign*} 
& {}_5F_4\l[ \begin{matrix} -n/2, & \frac{1-n}{2}, & 1 + \frac{c+bn}{1-2b}, & -a, & a + 1\\ 
3/2, & \frac{c+bn}{1-2b}, & c + 1, & 1 - c - n \end{matrix}\r] &\\
& \kern2.5cm =\; \frac{c(c+a+bn+b)}{(n+1)(2a+1)(c+bn)} \frac{(c + a + 1)_n}{(c)_n} %
\;-\; \frac{c(c-a-1+bn+b)}{(n+1)(2a+1)(c+bn)} \frac{(c-a)_n}{(c)_n}\,. &
\end{flalign*}
\end{prop}
\begin{proof}
Combine Theorem~\ref{hrrecirel1} with its variant under the parameter replacement that changes $a$ into $- 1 - a$ and then shift $n$ to $n + 1$. This gives yet another terminating balanced series identity where $b$ remains a free parameter, which completes the proof of the proposition.
\end{proof}
Proposition~\ref{prop:hrbalid2} is an extension of several known identities. If we let $b = 0$, Corollary~\ref{cor:chuwei2} in Subsection~\ref{cvgrecirel1} is recovered again. Similarly, the next Corollary~\ref{cor:twocor} displays two identities on ${}_4F_3$-series which can be established for two other particular values of $b$.

\begin{cor} \label{cor:twocor} 
\bi
\item[\,]
\item[(i)]\ If we let $b = \frac{a-c+1}{2a+n+2}$ in Proposition~\ref{prop:hrbalid2}) (see~\cite[Eq.~5.2b]{ChuWei08}), then
\begin{equation*} \arraycolsep=1.5pt
{}_4F_3\l[ \begin{matrix}-n/2, & \frac{1-n}{2}, & -a, & 2 + a\\3/2, & c + 1, & 1 - c - n \end{matrix}\r] %
\;=\; \frac{c}{2c+n)}\; \frac{(c+a+1)_{n+1} - (c-a-1)_{n+1}}{(n+1) (a+1) (c)_n}\,.
\end{equation*}

\item[(ii)]\ If we let $b \to 1/2$ in Proposition~\ref{prop:hrbalid2}) (see~\cite[Eq.~5.3b]{ChuWei08}), then
\begin{flalign*} 
& {}_4F_3\l[ \begin{matrix} -n/2, & \frac{1-n}{2}, & -a, & a + 1\\ 3/2, & c + 1, & 1 - c - n\end{matrix}\r] &\\
& \kern2cm =\; \frac{c (2c+2a+n+1)}{(n+1) (2a+1) (2c+n)} \frac{(c+a+1)_n}{(c)_n} %
\;-\; \frac{c (2c-2a+n-1)}{(n+1) (2a+1) (2c+n)} \frac{(c-a)_n}{(c)_n}\,.
\end{flalign*}
\ei
\end{cor}

\subsection{Second type of reciprocal relations for terminating balanced series} \label{hrreciprel2} \hfill \par
Substitute $c + \l[\frac{2n+2}{3}\r]$ for $c$ in Eq.~\eqref{eq:hrconv} and, again, apply the trinomial revision relation. Therefore,

\begin{flalign*}
\bin{c + \l[\frac{2n+2}{3}\r] - bk}{n - k} & =\; (-1)^{[n/2]}\, 2^{\l[\frac{2n+2}{3}\r]}\; %
\frac{(bk - k - c)_{[n/3]}}{(n-k)!\, (c - bk + 1)_k} &\\
& \times \l(\frac{c - bk + 2}{2}\r)_{\l[\frac{n+1}{3}\r]}\, \l(\frac{c - bk + 1}{2}\r)_{\l[\frac{n+2}{3}\r]}\,.
\end{flalign*}
Consequently, the Hagen--Rothe's formula can then express as

\begin{align*}
& \sum_{0\le k\le n} (-1)^k \bin{n}{k} (b k - k - c)_{[n/3]} %
\l(\frac{c - bk + 2}{2}\r)_{\l[\frac{n+1}{3}\r]}\, \l(\frac{c - bk + 1}{2}\r)_{\l[\frac{n+2}{3}\r]}\; \omega(k) &\\ 
& \kern3cm =\; (-1)^{[n/2]}\, 2^{-\l[\frac{2n+2}{3}\r]}\; \b(a + c + 1 - [n/3]\b)_n. &
\end{align*}

Specify now $\phi$, $\psi$ and $\chi$ to $\phi(x;n) := (bx - x - c)_n$, $\psi(x;n) := \l(\frac{1}{2}(c + 2 - bx)\r)_n$ and $\chi(x;n) := \l(\frac{1}{2}(c - bx + 1 )\r)_n$, respectively. The dual equation corresponding to Eq.~\eqref{hrsysF}-Eq.~\eqref{hrsysG} is
\begin{align*}
\omega(n) & =\; \sum_{k\ge 0} \bin{n}{3k}\; 2^{-2k}\; \frac{\l(\frac{c + 1 - 3bk}{2} + k\r) %
(a + c + 1 - k)_{3k}}{(bn - n - c)_{k} \l(\frac{1}{2} (c + 2 - bn)\r)_{k} \l(\frac{1}{2} (c + 1 - bn )\r)_{k + 1}}\\ 
& -\; \sum_{k\ge 0} \bin{n}{3k+1}\; 2^{-(2k+1)}\; \frac{\l(\frac{c - b + 2 - 3bk}{2} + k\r) %
(a + c + 1 - k)_{3k + 1}}{(bn - n - c)_{k} \l(\frac{c + 1 - bn}{2}\r)_{k+1} \l(\frac{c + 2 - bn}{2}\r)_{k+1}}\\
& +\; \sum_{k\ge 0} \bin{n}{3k+2}\; 2^{-(2k+2)}\; \frac{(3bk + 2b - c - 2 - 2k) (a + c + 1 - k)_{3k+2}} %
{(bn - n - c)_{k+1} \l(\frac{c + 1 - bn}{2}\r)_{k + 1} \l(\frac{c + 2 - bn}{2}\r)_{k+1}}\,.
\end{align*}

Changing $a + c$ into $-1 -a$ and $c + 2 - bn$ into $c + 1$ yields the next reformulation in terms of ${}_7F_6$-series, neither of which can be factorized into a closed form.

\begin{thm}\label{hrrecirel2} \emph{(Reciprocal relation)}
\begin{flalign*}
\frac{c + a + bn}{c + a}\; \frac{(c + a)_n}{(c)_n} & =\; \frac{c + bn}{c}\; \arraycolsep=1.2pt %
{}_7F_6\l[ \begin{matrix} -n/3, & \frac{1-n}{3}, & \frac{2-n}{3}, & 1+\frac{c+bn}{2-3b}, & -a/2, & \frac{1-a}{2}, & a + 1 \\ 
1/3, & 2/3, & \frac{c+bn}{2-3b}, & \frac{c+1}{2}, & \frac{c+2}{2}, & 1 - c - n \end{matrix}\r] &\\
& \kern-2.5cm +\; \frac{na (c + 1 + bn - b)}{c (c + 1)}\; \arraycolsep=1.2pt %
{}_7F_6\l[ \begin{matrix} \frac{1-n}{3}, & \frac{2-n}{3}, & \frac{3-n}{3}, & 1 + \frac{c + 1 + bn - b}{2-3b}, %
& \frac{1 - a}{2}, & \frac{2 - a}{2}, & a + 1\\ 
2/3, & 4/3, & \frac{c + 1 + bn - b}{2-3b}, & \frac{c+2}{2}, & \frac{c+3}{2}, & 1 - c - n \end{matrix}\r] &\\
& \kern-2.8cm +\; \frac{na (n - 1) ( a - 1) (c + 1 + bn - 2b)}{2c(c + 1)(c - 1 + n)}\; %
{}_7F_6\l[ \begin{matrix} \frac{2-n}{2}, \frac{3-n}{2}, \frac{4-n}{2}, 1 + \frac{c + 1 + bn - 2b}{2-3b}, %
\frac{2-a}{2}, \frac{3-a}{2}, a + 1 \\ 4/3, \ 5/3, \ \frac{c + 1 + bn - 2b}{2-3b}, %
\ \frac{c+2}{2}, \ \frac{c+3}{2}, \ 2 - c - n \end{matrix}\r]. &
\end{flalign*}
\end{thm}

\subsection{Third type of reciprocal relations for terminating balanced series} \label{hrreciprel3} \hfill \par
More generally, by replacing $c$ by $c + \l[\frac{(n+1)\ell}{2}\r]$ for any nonnegative integer $\ell$ in Eq.~\eqref{eq:hrconv} and making use of the trinomial revision identity, we obtain
\begin{flalign*}
\bin{c + \l[\frac{n+1}{2}\r] - b k}{n - k} & =\; \ell^{\l[\frac{\ell}{\ell+1}(n+1)\r]}\; %
\frac{(bk - k - c)_{\l[\frac{n}{\ell+1}\r]}}{(n-k)!\, (c + 1 - bk)_k} &\\
& \quad \times \; (-1)^{\l[\frac{n}{\ell+1}\r]}\; \prod_{1\le j\le \ell} %
\l(\frac{j+c-b k}{\ell}\r)_{\l[\frac{n+1+\ell-j}{\ell+1}\r]}.
\end{flalign*}
Therefore, the other next alternative to the Hagen--Rothe's formula in Eq.~\eqref{eq:hrconv} is

\begin{flalign*}
\sum_{0\le k\le n} (-1)^k \bin{n}{k} (bk - k - c)_{\l[\frac{n}{\ell+1}\r]}\; & %
\prod_{1\le j\le \ell} \l(\frac{j + c - bk}{\ell}\r)_{\l[\frac{n + 1+ \ell - j}{\ell+1}\r]}\; \omega(k) &\\
& =\; (-1)^{\l[\frac{n}{\ell+1}\r]}\; \ell^{-\l[\frac{\ell}{\ell+1}(n+1)\r]}\; %
\l(a + c + 1 -\l[\frac{n}{\ell+1}\r]\r)_n. &
\end{flalign*}

Specify now $\phi_0(x;n)$ to $(bx - x - c)_n$ and $\phi_j(x;n)$ to $\l(\Frac{\ell - j + c - bx + 1}{\ell}\r)_n$ with $1\le j\le \ell$, the identity coincides with Eq.~\eqref{sysF}; and the dual relation corresponding to~\eqref{sysG} is

\begin{flalign*}
\omega(n) & =\; \sum_{k\ge 0} \bin{n}{(\ell+1)k + \ell}\, \frac{(-1)^{k\ell+\ell}}{\ell^{k\ell+\ell}}\, %
\frac{(a + c + 1 - k)_{(\ell+1)k+\ell}}{(bn - n - c)_{k+1}}\, \frac{k + (b-1)\b((\ell+1)k + \ell\b) - c} %
{\Prod_{1\le j\le \ell} \l(\frac{\ell - j + c - bn + 1}{\ell}\r)_{k+1}} &\\
& -\; \sum_{i=1}^\ell \sum_{k\ge 0} \bin{n}{(\ell+1)k + i - 1}\; %
\frac{(a + c + 1 - k)_{(\ell+1)k + i - 1}}{(bn - n - c)_{k}} &\\
& \kern4.5cm \times \; \frac{(-1)^{k\ell+i}}{\ell^{k\ell+\l[\frac{i\ell}{\ell+1}\r]+1}}\; %
\frac{k\ell + i + c - b\b((\ell+1)k + i - 1\b)} %
{\Prod_{1\le j\le \ell} \l(\frac{\ell - j + c - bn + 1}{\ell}\r)_{k+\l[\frac{i+j}{\ell+1}\r]}}\,. &
\end{flalign*}

Finally, change again the parameters $a + c$ into $- 1 -a$ and $c - bn + 2$ into $c + 1$. After a few simplification (e.g. such as $(1)_\ell = \ell !$ or $(-x)_\ell = (\ell - x - 1)^{\und{\ell}}$) the equality writes in terms of hypergeometric series as

\begin{thm} \label{hrrecirel3} \emph{(Reciprocal relation)}
\begin{flalign*} 
\frac{c + a + bn}{c + a}\, \frac{(c + a)_n}{(c)_n} & \;=\; \frac{(-n)_{\ell} (-a)_{\ell}}{\ell!\, (1 - c - n)}\; %
\frac{(b - 1)\ell - c - bn + 1}{\Prod_{1\le j\le \ell} (\ell - j + c)}\; \times &\\
& \kern-2cm \arraycolsep=2.2pt {}_{2\ell+3}F_{2\ell+2}\l[ \begin{matrix} a + 1, & 1 + \frac{c + bn - (b-1)\ell - 1} %
{\ell -b(\ell+1)}, & \l\{\frac{\ell + j - n}{\ell+1}\r\}_{j=0}^\ell, & \l\{\frac{\ell + j - a - 1}{\ell}\r\}_{j=1}^\ell\\ 
2 - c - n, & \frac{c + bn - (b-1)\ell - 1}{\ell - b(\ell+1)}, %
& \l\{\frac{\ell+j+1}{\ell+1}\r\}_{j=1}^\ell, & \l\{\frac{2\ell-j+c}{\ell}\r\}_{j=1}^\ell \end{matrix}\r] &\\
& \kern-3cm +\; \sum_{i=1}^\ell \frac{(-n)_{i-1}}{(i-1)!}\; \frac{(-a)_{i-1} \b(c + bn - (b - 1)(i - 1)\b)} %
{\Prod_{1\le j\le \ell}\, \Prod_{1\le k\le \l[\frac{i+j}{\ell+1}\r]} (k\ell - j + c)}\; \times &\\
& \kern-2cm {}_{2\ell+4}F_{2\ell+3}\l[ \begin{matrix} 1, \ a + 1, \ 1 + \frac{c + bn - (b-1)(i-1)} {\ell - b(\ell+1)}, %
\ \l\{\frac{i + j - n - 1}{\ell+1}\r\}_{j=0}^\ell, \ \l\{\frac{i+j-a-2}{\ell}\r\}_{j=1}^\ell \\ 
1 - c - n, \quad \frac{c + bn - (b-1)(i-1)}{\ell-b (\ell+1)}, \qquad \l\{\frac{i+j}{\ell+1}\r\}_{j=0}^\ell, %
\qquad \l\{\frac{\ell - j + c}{\ell} + \l[\frac{i+j}{\ell+1}\r]\r\}_{j=1}^\ell \end{matrix}\r]. &
\end{flalign*}
\end{thm}

\section{Perspective}
Following the same lines as in the paper, the triplicate and multiplicate inversions performed from the Chu--Vandermonde--Gau\ss's, the Pfaff--Saalsch\"utz's and the Hagen--Rothe's summation formulae can also enable to produce still more reciprocal relations on terminating hypergeometric series. Such new identities will take more and more involved forms, e.g. especially space consuming. However, generalizing to any value of the parameters---such as $\ell$ in Theorem~\ref{hrrecirel3}---, should enhance the relevance of the multiplicate inversions of the ``third approach'' and enlarge the viewpoint of the present purpose.

\bibliographystyle{article}
\def\bibfmta#1#2#3#4{ {\sc #1}, {#2}, \emph{#3} #4.}
\bibliographystyle{book}
\def\bibfmtb#1#2#3#4{ {\sc#1}, \emph{#2}, {#3}, #4.}

\vskip 1cm 
\no {\small Christian {\sc Lavault}
\newline Université Paris 13, Sorbonne Paris Cité, 
\newline Laboratoire d'Informatique de Paris-Nord F-93430 Villetaneuse.
\newline (LIPN, CNRS UMR 7030 -- \texttt{http://lipn.univ-paris13.fr})
\newline \emph{E-mail:} \url{Christian.Lavault@lipn.univ-paris13.fr}, 
\newline \emph{URL:} \url{http://lipn.univ-paris13.fr/\textasciitilde lavault}
}

\end{document}